\documentclass[10pt,twoside,a4paper]{amsart}

\usepackage[english]{babel}
\usepackage[utf8]{inputenc}

\usepackage{amsfonts,amsthm,amssymb,amsmath,amsthm,latexsym,wasysym,stmaryrd,mathrsfs,dsfont,txfonts,xcolor}
\usepackage{accents,multirow,rotating,subfigure,graphicx,bigstrut,lscape,multicol,enumerate,accents,mathtools,exscale}

\usepackage[normalem]{ulem}

\usepackage{pdftricks}
\usepackage{color}
\usepackage[pdftex,all]{xy}

\usepackage{hyperref}
\usepackage{appendix}
\usepackage[justification=centering]{caption}

\graphicspath{{Figures/}}

\numberwithin{equation}{section}

\theoremstyle{theorem}

\newtheorem {theo}{Theorem}[section]
\newtheorem*{theo*}{Theorem}
\newtheorem {lem}[theo]{Lemma}
\newtheorem*{lem*}{Lemma}
\newtheorem {prop}[theo]{Proposition}
\newtheorem*{prop*}{Proposition}
\newtheorem {cor}[theo]{Corollary}
\newtheorem*{cor*}{Corollary}

\newtheorem*{cor_proof*}{Corollary (of the proof)}

\newtheorem*{conjecture*}{Conjecture}
\theoremstyle{definition}
\newtheorem {defi}[theo]{Definition}
\newtheorem*{defi*}{Definition}

\newtheorem*{nota*}{Notation}
\theoremstyle{remark}
\newtheorem {rem}[theo]{Remark}
\newtheorem*{rem*}{Remark}

\newtheorem*{warning*}{Warning}

\newtheorem {convention}[theo]{Convention}
\newtheorem*{convention*}{Convention}

\newtheorem*{ex*}{Example}

\newtheorem*{question*}{Question}

\newtheorem*{questions*}{Questions}

\newtheorem*{fact*}{Fact}
\newtheorem{claim}[theo]{Claim}

\def\C{{\mathcal C}}
\def\CC{{\mathcal C}}

\def\R{{\mathds R}}

\def\UU{\mathcal C}

\def\Z{{\mathds Z}}

\def\2Z{{\fract{\Z}{2\Z}}}

\def\e{\varepsilon}
\def\p{\partial}

\newcommand{\fract}[2]{\hbox{\leavevmode 
\kern.1em \raise .25ex \hbox{\the\scriptfont0 $#1$}\kern-.1em }\big/
  {\hbox{\kern-.15em \lower .5ex \hbox{\the\scriptfont0 $#2$}} }}
\newcommand{\subfract}[2]{\hbox{\leavevmode
  \kern0em \raise .25ex \hbox{\the\scriptfont0 \tiny $#1$}\kern-.1em }/
  {\hbox{\kern-.15em \lower .5ex \hbox{\the\scriptfont0 \tiny $#2$}} }}

\newcommand{\dessin}[2]{
  \vcenter{\hbox{\includegraphics[height=#1]{#2.pdf}}}}

\def\nR{\textnormal R}

\definecolor{pink}{rgb}{0.858, 0.188, 0.478}
\definecolor{orange}{rgb}{1, 0.647, 0}
\definecolor{vert}{rgb}{0.42, 0.557, 0.137}

\begin{document} 

\title{Combinatorial link concordance using cut-diagrams} 
\author[B. Audoux]{Benjamin Audoux}
         \address{Aix Marseille Univ, CNRS, Centrale Marseille, I2M, Marseille, France}
         \email{benjamin.audoux@univ-amu.fr}
\author[J.B. Meilhan]{Jean-Baptiste Meilhan} 
\address{Univ. Grenoble Alpes, CNRS, Institut Fourier, F-38000 Grenoble, France}
	 \email{jean-baptiste.meilhan@univ-grenoble-alpes.fr}
\author[A. Yasuhara]{Akira Yasuhara} 
\address{Faculty of Commerce, Waseda University, 1-6-1 Nishi-Waseda,
  Shinjuku-ku, Tokyo 169-8050, Japan}
	 \email{yasuhara@waseda.jp}
\begin{abstract} 
Cut-diagrams are diagrammatic objects, defined in dimensions $1$ and $2$, that generalize links in $3$-space and surface-links in $4$-space; in dimension $1$, this coincides with the theory of welded links. 
Using cut-diagrams, we introduce an equivalence relation called cut-concordance, which encompasses the topological notion of concordance for classical links. Our main result is that the nilpotent peripheral  system of $1$--dimensional cut-diagrams is an invariant of cut-concordance, giving along the way a combinatorial version of a theorem of Stallings. 
We also investigate the relationship with several other equivalence relations in diagrammatic knot theory, in particular in connection with link-homotopy.
\end{abstract} 

\maketitle

\section*{Introduction}
A central feature of knot theory is its combinatorial nature: knots and links can be represented by planar diagrams, and ambient isotopy corresponds to a finite set of local (Reidemeister) moves. This diagrammatic viewpoint has multiple advantages, and led in particular to a variety of combinatorial extensions of classical knot theory, among which welded knot theory. 
Welded links arise indeed as a generalization of (classical) link diagrams where one allows virtual crossings in addition to classical ones, subject to specific local moves; quite remarkably, it has  provided in recent years a significant number of applications in both $3$--dimensional \cite{Graff,Colombari,AM,AMY_k} and $4$--dimensional topology \cite{ABMW,arrow,AMY_w}. 

Welded theory can be equivalently reformulated using the notion of \emph{cut-diagrams} introduced in \cite{AMY}. 
But the notion of cut-diagrams also exist in dimension $2$, where they extend surfaces embedded in $4$--space. In both settings, cut-diagrams encode and generalize combinatorially the topology of embedded objects, and important algebraic invariants such as the peripheral system. Cut-diagrams hence provide a convenient framework for studying classical topological phenomena using purely combinatorial tools.

Beyond ambient isotopy, another fundamental equivalence relation in knot theory is \emph{concordance}. Two links are concordant if they cobound a collection of disjoint, smoothly embedded annuli in $S^3 \times [0,1]$. 
The aim of the present paper is to introduce and study a diagrammatic notion of concordance for cut-diagrams, which we call \emph{cut-concordance}. Roughly speaking, a $2$--dimensional cut-diagram may interpolate between two $1$--dimensional cut-diagrams on its boundary components, providing a combinatorial analogue of a concordance. 
In particular, a concordance between two links induces a cut-concordance between the cut-diagrams associated to both links; in this sense, cut-concordance extends classical concordance, while remaining entirely within the combinatorial framework of cut-diagrams.

We observe that, in strong contrast with the classical case, any knot is \lq slice\rq\, in this  setting, i.e. is cut-concordant to the unknot (Lemma \ref{lem:cutslice}). 
We next prove our main invariance result: 
\begin{theo*}[Thm.~\ref{th:NilpotentConcordance}]
  Equivalence classes of nilpotent peripheral systems for $1$--dimensional cut-diagrams are invariant under cut-concordance.
\end{theo*}
Recall that the $q$-th nilpotent quotient of a link is the quotient of the fundamental group of the exterior by the $q$-th term of its lower central series ($q\ge 1$), and that the $q$-th \emph{nilpotent peripheral system} is the data of this quotient together with (the image of) so-called peripheral elements, which are represented by a meridian and preferred longitude for each component. Nilpotent peripheral systems extend naturally to all $1$--dimensional cut-diagrams. 

As a consequence of Theorem \ref{th:NilpotentConcordance}, all invariants derived from the peripheral system, such as Milnor invariants, remain unchanged under cut-concordance (Corollary \ref {cor:MilnorConcordance}). 
Another noteworthy consequence of our proof, given in Section \ref{sec:onsamusebiennon}, is a purely  combinatorial proof and generalization of a theorem of Stallings \cite[Thm.~5.2]{Stallings} in the case of a topological link concordance (Proposition \ref{rem:Stallings}). 
\medskip 

In addition, we investigate in Section \ref{sec:arrementnimp} the relationship between cut-concordance and other equivalence relations that appear in diagrammatic knot theory. In particular, we compare it with welded concordance \cite{BC} and with analogues of link-homotopy such as the $sv$-equivalence generated by self-virtualization moves \cite{AM,ABMW} and the reduced cut-concordance for cut-diagrams (Section \ref{sec:4.2.1}). 
The following commutative diagram summarizes how cut-concordance fits into this broader hierarchy of equivalence relations on classical, welded, and diagrammatic knotted objects, where we use the standard notations for surjections (double-headed arrow) and injections (hooked arrow tail): 
\[
  \xymatrix@R=3pc@C=4pc{
   \dfrac{\big\{\textrm{links}\big\}}{\textrm{concordance}}\ar[r] \ar@{->>}[d]_{\textrm{\cite{Goldsmith,Giffen}}}  & \dfrac{\big\{\textrm{welded links}\big\}}{\textrm{welded concordance}}\ar@{->>}[r]^{\textrm{Lem.~\ref{prop:wctocut}\quad\qquad}}\ar@{->>}[d]^{\textrm{Rem.~\ref{lem:niscatedebernoulli}}} & \dfrac{\big\{\textrm{$1$--dim. cut-diagrams}\big\}}{\textrm{cut-concordance}}\ar@{->>}[d]^{\textrm{Rem.~\ref{lem:niscatedebernoulli}}} \\
   \dfrac{\big\{\textrm{links}\big\}}{\textrm{link-homotopy}}\ar@{^{(}->}[r]^{\textrm{\cite{AM,AMY_w}}\quad} & \dfrac{\big\{\textrm{welded links}\big\}}{\textrm{sv-equivalence}}\ar@{^{(}->>}[r]^{\textrm{Lem.~\ref{lem:svtocut}\quad\qquad}}_{\textrm{Cor.~\ref{lem:cuttosv}\quad\qquad}} 
   & \dfrac{\big\{\textrm{$1$--dim. cut-diagrams}\big\}}{\textrm{reduced cut-concordance}}
  }
\]
\medskip

We stress that, although this introduction only deals with knots and links, \emph{all} results in this paper are actually given in the broader context of tangles. 
Recall that a \emph{tangle with skeleton $X$}, where $X$ is a compact, oriented and ordered $1$-manifold, is the image of a proper embedding of $X$ in the $3$-ball, such that the boundary points of $X$ are mapped to a fixed set of ordered marked points on the boundary $2$-sphere, up to ambient isotopy fixing the endpoints. Additional boundary conditions may be imposed: for example, a string link is a tangle with skeleton a union of intervals, such that the $i$-th interval runs, say, from the $(2i-1)$-th to the $2i$-th marked point.

\section{Cut-diagrams}  \label{sec:kut}

In this section we quickly review the theory of cut-diagrams developed in \cite{AMY,AMY_kirk}.

\subsection{$1$--dimensional cut-diagrams and welded tangles}\label{sec:1kut}

\subsubsection{$1$--dimensional cut-diagrams}
Let $X$ be a compact oriented $1$-manifold. 
\begin{defi}
A \emph{$1$--dimensional cut-diagram} over $X$ is a collection of pairwise distinct points on $X$, called \emph{cut-points}, endowed with a sign and labeled by some region, where the \emph{regions} are the connected component of $X$ with all cut-points removed. 
\end{defi}
\noindent We stress that in this definition there is no constraint on the labeling. 
\medskip 

Given a tangle diagram $D$, which is a (generic) proper immersion $X\looparrowright \mathbb{D}^2$ in the $2$-disk $\mathbb{D}^2$ with an over/under decoration at each double point, there is an associated cut-diagram $C_D$ over $X$ defined as follows. For each crossing, consider the preimage with lowest coordinate in the projection axis: this collection of points forms the cut-points of $C_D$. 
Each cut-point is naturally labeled by the sign of the
corresponding crossing, and by the region associated with the overpassing strand. 
See Figure \ref{fig:1DimCut} for examples. 
\begin{figure} 
  \[
    \dessin{2.3cm}{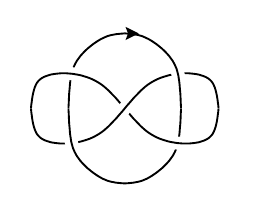} \ \leadsto \, \dessin{2.7cm}{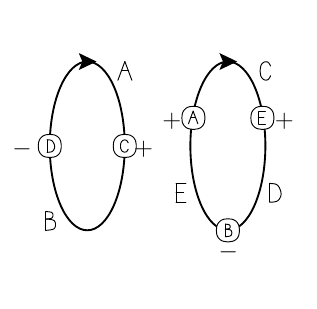}\quad ; \quad
\dessin{2.5cm}{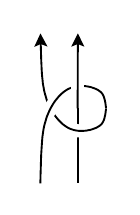} \ \leadsto\, \dessin{2.7cm}{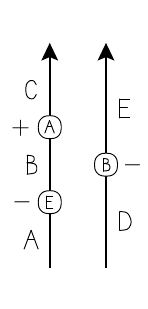}
  \]
  \caption{$1$--dimensional cut-diagrams arising from tangle diagrams.\\{\footnotesize Regions are
      named with capital letters, and labels on cut-points are given by circled nametags}}
  \label{fig:1DimCut}
\end{figure}

A $1$--dimensional cut-diagram arising in this way will be referred to as \emph{topological}. 

\begin{rem}\label{rem:kutGD}
Many $1$--dimensional cut-diagrams are not topological. 
This is reminiscent of Gauss diagrams \cite{GPV}, which often are not realizable by a link or tangle. As a matter of fact, these notions are closely related. 
Indeed, a cut-point of a $1$--dimensional cut-diagram can be seen an arrow head in a Gauss diagram, and the label represents the region where the arrow tail is attached. 
The only difference is that a cut-diagram does not specify the relative position of adjacent tails: cut-diagrams can thus be seen as Gauss diagrams up to the \emph{Tail-Commute move}, which is the local move swapping two adjacent arrow tails. 
\end{rem}

\subsubsection{Topological moves and welded tangles}

The three \emph{$1$--dimensional topological moves} given in Figure \ref{fig:IsotopyDim1} are the straightforward translations of each of the three Reidemeister moves of knot theory into the language of $1$--dimensional cut-diagrams. 
\begin{figure}
\[
\dessin{2cm}{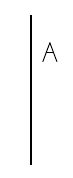}\ \stackrel{\textrm{R1}}{\longleftrightarrow}\ \dessin{2cm}{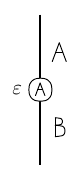}
\hspace{1.5cm}
\dessin{2cm}{Dim1Move1_1}\ \stackrel{\textrm{R2}}{\longleftrightarrow}\ \dessin{2cm}{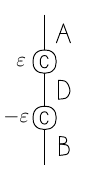}
\hspace{1.5cm}
\dessin{2cm}{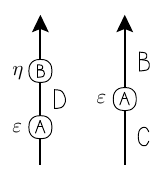}\ \stackrel{\textrm{R3}}{\longleftrightarrow}\ \dessin{2cm}{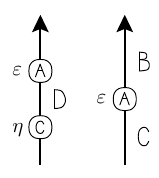}
\]
  \caption{Topological moves on $1$--dimensional cut-diagrams ($\e,\eta=\pm$):\\{\footnotesize
      in the middle and right moves, region $D$ must not occur as the
      label of any cut-point; in the left/middle move,
      when going from left to right,
      every $A$--label becomes either an $A$ or
      $B$--label, and when going from right to left, all the
      $A$ and $B$--labels become
      $A$--labels}}
\label{fig:IsotopyDim1}
\end{figure}
Two $1$--dimensional cut-diagrams are called \emph{equivalent} if they differ by a sequence of these moves. 
\medskip

Recall that a \emph{virtual diagram} with skeleton $X$ is a generic proper immersion of the $1$-manifold $X$ in the $2$-disk, whose transverse double points are decorated as either classical crossings (as in usual knot diagrams) or virtual crossings (depicted by a circled double point).  
A \emph{welded tangle} with skeleton $X$ is the equivalence class of a virtual diagram (with skeleton $X$) modulo \emph{welded Reidemeister moves}: this is the set of local moves given by the usual Reidemeister moves, the virtual Reidemeister moves (where all classical crossings in Reidemeister moves are replaced by virtual ones), 
together with the Mixed and the Over-Commute (OC) moves given in  Figure \ref{fig:wMoves}. 
In particular, \emph{welded links  (resp. string links)} are equivalence classes of virtual diagrams with skeleton a union of circles (resp. of intervals).  
\begin{figure}
\[
    \begin{array}{cc}
      \dessin{1.5cm}{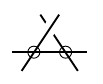}\stackrel{\textrm{Mixed}}{\longleftrightarrow} \dessin{1.5cm}{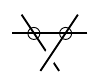} \, & \, 
      \dessin{1.5cm}{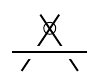}\ \stackrel{\textrm{OC}}{\longleftrightarrow}
      \dessin{1.5cm}{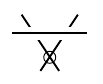}
    \end{array}
\]
  \caption{The mixed and OC moves}\label{fig:wMoves}
\end{figure}

Now, as noticed in Remark \ref{rem:kutGD}, $1$--dimensional cut-diagrams are in one-to-one correspondence with Gauss diagrams up to Tail-Commute moves. 
Since Gauss diagrams up to Reidemeister and Tail-Commute moves exactly
correspond to welded tangles (see eg. \cite[\S~4.5]{ABMW}),
$1$--dimensional cut-diagrams up to topological moves are in fact just
a reformulation of the welded theory. More precisely, we have the
following.
\begin{prop}
Equivalence classes of $1$--dimensional cut-diagrams over a $1$-manifold $X$ are in bijection with welded tangles with skeleton $X$.   
\end{prop}
Since two diagrams of a same tangle are connected by a finite sequence
of Reidemeister moves, and since $1$--dimensional topological moves are
the direct translations of these moves, we have a well-defined map
\[
  \Xi_1:\big\{\textrm{tangles}\big\}\to \big\{\textrm{equivalence classes of
    $1$--dimensional cut-diagrams}\big\}.
\]

\subsection{$2$--dimensional cut-diagrams and surface-links}\label{sec:kut2}

Let $\Sigma$ be a compact oriented surface, possibly with boundary.

\subsubsection{$2$--dimensional cut-diagrams}

A diagram on $\Sigma$ is a compact oriented (generically, but not necessarily properly)
immersed $1$-manifold $P$ in $\Sigma$, together with an over/under decoration at each crossing.  The diagram $P$  splits $\Sigma$ into connected components called \emph{regions}, 
  and the under crossings split $P$ into \emph{cut-arcs}. 

\begin{defi}\label{def:cut}
  A \emph{$2$--dimensional cut-diagram} over $\Sigma$ is a diagram on $\Sigma$, endowed with a labeling of each cut-arc by a region, satisfying the following \emph{labeling rules}: 
 \begin{itemize}
  \item[1)] for each crossing involving labels $A,B,C$ as shown on the left-hand side of Figure \ref{fig:cond}, the regions $A$ and $B$ are adjacent along a $C$-labeled cut-arc as illustrated in the figure; 
  \item[2)] a cut-arc containing a univalent vertex in region $A$, is labeled by $A$ (see the right-hand side of Figure \ref{fig:cond}). 
  \end{itemize}
 \end{defi}
\begin{figure}
  \[ \begin{array}{ccc}
   \dessin{2cm}{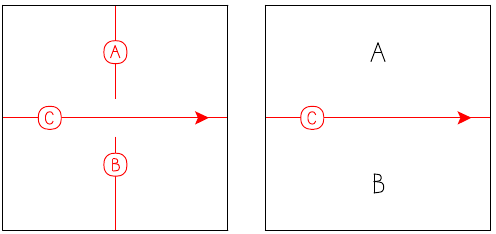} & & \quad\dessin{2cm}{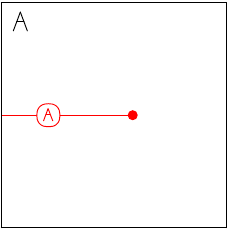} \\[0.5cm]
   \textrm{First rule} & & \quad\textrm{Second rule} 
   \end{array}\]
  \caption{Labeling rules for $2$--dimensional cut-diagrams} 
  \label{fig:cond}
\end{figure}
In figures, we shall represent cut-arcs in red to avoid any confusion with $\partial \Sigma$.
\medskip 

Any surface-link in $4$-space yields a ($2$--dimensional) cut-diagram. 
This arises from the notion of \emph{broken surface diagrams},\footnote{Here we shall simply called these \lq surface diagrams\rq.} (recalled below), which are natural analogues of knot diagrams for surface-links, see \cite{CS}. 
Recall that a \emph{surface-link} is a proper smooth embedding of a compact oriented surface in the $4$--dimensional ball; two surface-links are \emph{equivalent} if their images are ambient isotopic relative to boundary.

Consider a surface-link $S$, given by a smooth embedding of the surface $\Sigma$ into $\mathbb{R}^4$. 
A  surface diagram of $S$ is given by a generic  projection from $\mathbb{R}^4$ to $\mathbb{R}^3$. 
The resulting immersed surface contains lines of transverse double points, which may meet at triple points and/or end at branch points. 
Double points are endowed with an over/under information using the projection axis, which is encoded by cutting off a neighborhood of the lowest preimage. 
The local models for triple points and branch points are given on the left-hand side of Figure \ref{fig:local}. 
Each line of double points also has a natural orientation, as follows: the local frame given by a positive normal vector to the \lq over region\rq, a positive normal vector to the \lq under region\rq, and a positive tangent vector to the double point line, must agree with the ambient orientation of $\R^3$. 
\begin{figure}
   \[\begin{array}{ccc}
  \dessin{2.7cm}{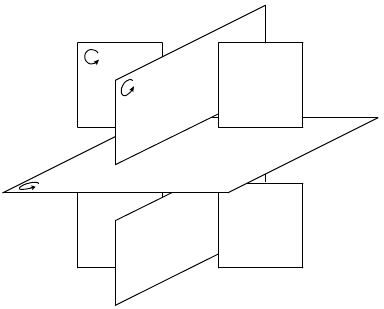} & \leadsto & \dessin{2.7cm}{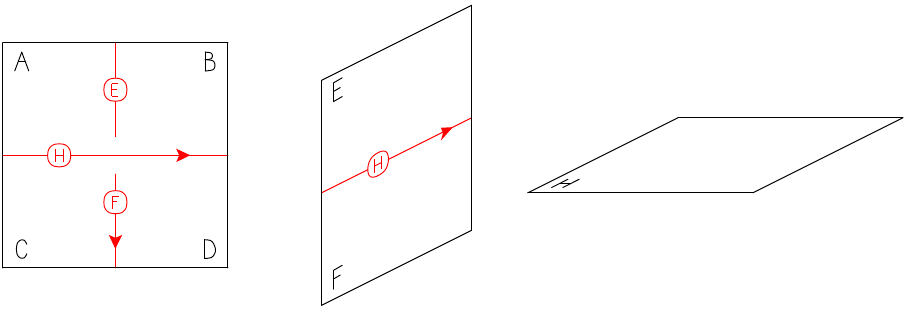}    \end{array}\]
   \[
  \begin{array}{ccc}
  \dessin{2cm}{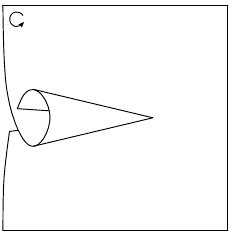} & \leadsto & \dessin{2cm}{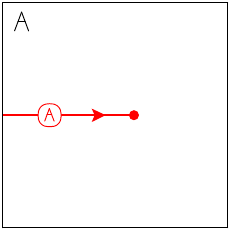}
   \end{array}\]
\caption{Local models for triple points and branch points in a  surface diagram (left), and the associated local cut-diagrams (right)}
  \label{fig:local}
\end{figure}

A $2$--dimensional cut-diagram $D_S$ on $\Sigma$ is naturally associated with a 
surface diagram, as follows. 
Consider on the surface $\Sigma$ the preimages of all double points:  the \emph{lower point set} is a union $P$ of oriented immersed circles and/or intervals, 
and each triple point of the surface diagram provides an over/under information at the corresponding crossing of $P$,  while each branch point produces an internal point of $P$, see Figure \ref{fig:local}.
We label each arc of the resulting diagram by the region containing the preimage with highest coordinate at the corresponding line of double points. As Figure \ref{fig:local} illustrates, this labeling automatically satisfies the labeling rules of Definition \ref{def:cut}. 

A $2$--dimensional cut-diagram arising in this way, from a surface-link in $4$--space, is called \emph{topological}.

\subsubsection{Topological moves for $2$--dimensional cut-diagrams}\label{sec:kut2.2}

The above procedure associates a $2$--dimensional cut-diagram to a fixed surface diagram of a surface-link. 
But there is a Reidemeister-type theorem for surface diagrams, due to D.~Roseman: 
two surface diagrams represent equivalent surface-links if and only if they differ by a sequence of the seven \emph{Roseman moves} given in \cite{Roseman3}. These Roseman moves can be reformulated into the language of cut-diagrams, producing a set of \emph{topological moves} for $2$--dimensional cut-diagrams. Such a set of moves $T_1$, ... , $T_7$ is given in \cite[Fig.~3.4]{AMY_kirk} (an equivalent set of moves, called cut-moves, can  be found in \cite[Fig.~6]{AMY}). 

Two $2$--dimensional cut-diagrams are \emph{equivalent} if they differ by a sequence of topological moves. 
 By Roseman's theorem \cite[Thm.~1]{Roseman3}, two topological
 cut-diagrams of equivalent surface-links are hence equivalent; 
 see \cite{AMY,AMY_kirk}. We thus have a well-defined map
\[
  \Xi_2:\big\{\textrm{surface-links}\big\}\to \big\{\textrm{equivalence classes of
    $2$--dimensional cut-diagrams}\big\}.
\]

\medskip 

In this paper, we shall only need the three topological moves T1, T3 and T5 shown in Figure \ref{fig:topomoves}.
\begin{figure} 
  \[
  \hspace{-.75cm}  \begin{array}{ccccc} 
     \dessin{1.7cm}{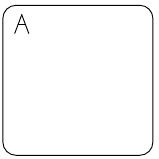} \stackrel{\textrm{T1}}{\longleftrightarrow} 
     \dessin{1.7cm}{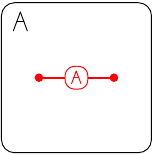} & & 
     \dessin{1.7cm}{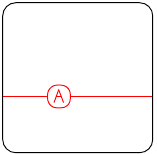} \stackrel{\textrm{T3}}{\longleftrightarrow} 
     \dessin{1.7cm}{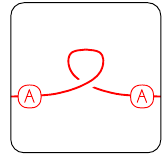}  & & 
     \dessin{1.7cm}{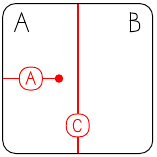} \stackrel{\textrm{T5}}{\longleftrightarrow} 
     \dessin{1.7cm}{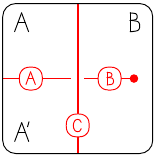}  
  \end{array}
  \]
  \caption{Three topological moves for $2$--dimensional cut-diagrams:\\ 
  \footnotesize{different letters for regions or arc labels may refer to
      the same region. 
  }}
\label{fig:topomoves}
\end{figure}
We stress that move T5 may split a region $A$ into two regions $A$ and $A'$, or conversely merge two existing regions. In the former case, all $A$-labeled cut-arcs are relabeled by either $A$ or $A'$, in such a way that the labeling rules of Definition \ref{def:cut} are fulfilled; in the latter case, all cut-arcs labeled by $A$ or $A'$ become $A$-labeled. 

\section{Group and nilpotent peripheral systems of cut-diagrams}

In this section, let $\UU$ be a cut-diagram over $M$, where $M=M_1\cup \cdots \cup M_n$ is a compact oriented manifold of dimension either $1$ or $2$. 

\subsection{The group of a cut-diagram}

The \emph{group} $G(\mathcal{C})$ of $\mathcal{C}$ is the group 
generated by its regions, with relations as follows:  
\begin{itemize}
\item if dim$M=1$, a relation $B^{-1}C^{-1}AC$ for each pair of regions $(A,B)$ that are adjacent at a $C$-labeled cut-point as in one of the two configurations shown on the left-hand side of Figure \ref{fig:adj};
\item if dim$M=2$, a relation $B^{-1}C^{-1}AC$ for each pair of regions $(A,B)$ that are adjacent along a $C$-labeled cut-arc as on the right-hand side of Figure \ref{fig:adj}.
\end{itemize}
\begin{figure}
\[\dessin{2.5cm}{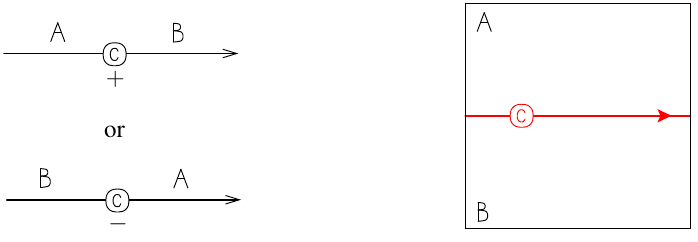}\]
 \caption{Configurations giving the relation $B^{-1}C^{-1}AC$ in the group of a $1$--dimensional (left) or $2$--dimensional (right) cut-diagram}
\label{fig:adj}
\end{figure}
An \emph{$i$-th meridian} is a region of $M_i$ ($i=1,\ldots,n$) when regarded as a generator of  $G(\mathcal{C})$. 
Note that any two $i$-th meridians are always conjugate in $G(\CC)$. 

\begin{prop}\label{prop:iniscent}
Let $\UU'$ be a cut-diagram obtained from $\UU$ by some topological move. 
There is a canonical isomorphism from $G(\UU)$ to $G(\UU')$, which induces the identity on the region labeling of any cut-point/cut-arc that is not involved in the move. 
\end{prop}
\begin{proof}
A proof for the case of the
  dimension $2$ is given in  \cite[Lem.~3.11]{AMY_kirk}, which
  adapts easily to the $1$--dimensional case.
\end{proof}

We note the following immediate consequence of the Wirtinger algorithm (see \cite[Prop.~2.2]{AMY} for the $2$--dimensional case). 
\begin{prop}\label{prop:ortion}
If $T$ is a tangle, the fundamental group $\pi_1(B^3\setminus T)$ of its complement is isomorphic to the group $G\big(\Xi_1(T)\big)$ of any associated $1$--dimensional cut-diagram; likewise for a surface-link $S$ we have $\pi_1(B^4\setminus S)\cong G\big(\Xi_2(S)\big)$. In both case, topological meridians are sent to cut-diagram meridians by this isomorphism.
\end{prop}

Now, let $\gamma$ be an oriented path on $M$;  
in the $1$--dimensional case, $\gamma$ is given the orientation induced by $M$. 
We may freely assume that $\gamma$ intersects $\mathcal{C}$ transversally at a finite number of regular points.  We associate elements $\widetilde{w}_\gamma$ and $w_\gamma$ in $G(\mathcal{C})$ as follows. 
For the $i$-th intersection point between $\gamma$ and $\mathcal{C}$ met when running along $\gamma$ according to its orientation,  
denote by $R_i$ the label of the cut-point/cut-arc met at this point, and by $\varepsilon_i$ the 
sign of the cut-point/intersection point.
Then 
\[
\widetilde{w}_\gamma:=R_1^{\varepsilon_1}\cdots
R_{k}^{\varepsilon_{k}} \quad \textrm{ and } \quad 
 w_\gamma:=R^{-\vert \gamma \vert}\widetilde{w}_\gamma, \]
\noindent where $R$ is the starting region of $\gamma$, $k=|\gamma\cap\mathcal{C}|$ is the number of intersection points between $\gamma$ and $\mathcal{C}$, and $\vert \gamma\vert$  is the sum of the exponents $\varepsilon_i$ in $\widetilde{w}_{\gamma}$ such that $R_i$ is in the same connected component as $R$. 
\\
Note that if $\gamma$ and $\gamma'$ are two homotopic paths, relative to boundary, on $M$ 
then $w_\gamma=w_{\gamma'}$ in $G(\mathcal{C})$; see \cite[Lem.~2.8]{AMY}.  

Now, for each $i$, pick a basepoint $p_i$ in the interior of some region of $M_i$; 
this specifies a \emph{preferred meridian} on each component of $M$. 
\begin{defi}\label{def:longitudes}
A \emph{system of loop-longitudes} for $M_i$ is a  
collection of words $w_{ij}:= w_{\gamma_{ij}}$ in $G(\UU)$, where 
$\{\gamma_{ij}\}_j$ is a collection of loops based at $p_{i}$ representing a generating set for $\pi_1(M_i;p_i)$. 
\end{defi}

\begin{rem}\label{rem:long1}
In the $1$--dimensional case, several choices are canonical. 
On the one hand, for each $i$ such that $M_i\cong [0,1]$, there is a canonical basepoint, in the (interior of the) first region met following the orientation. 
On the other hand, given a choice of basepoint for each $S^1$-component of $M$, there is a canonical system of loop-longitudes given, for each $i$ such that $M_i\cong S^1$, by the based loop running along $M_{i}$ following the orientation.
\end{rem}

\subsection{Nilpotent quotients and the Chen maps}
A convenient feature of the group of a cut-diagram $\C$, is that its
\lq nilpotent quotients\rq, defined below, have a presentation with only one
generator per component of $M$, and relations being either iterated commutators or meridian/longitude commutations; moreover, these longitudes are algorithmically computed using
the so-called \lq Chen maps\rq. We review now all these notions.

Here and throughout the rest of this paper, the commutator of two elements $a,b$ of some group is defined as $[a,b]:=a^{-1}b^{-1}ab$.  
\begin{defi}
The \emph{lower central series} $\left( G_q\right)_{q\ge 1}$ of a group $G$, 
is the descending series of subgroups inductively defined by $G_1=G$ and $G_{q+1}=[G,G_q]$. For $q\ge 1$, the \emph{$q$-th nilpotent quotient} of $G$ is the quotient $N_qG:=\fract{G}{G_q}$. 
\end{defi}

Returning to the cut-diagram $\UU$ over $M$, we can thus define for $q\ge 1$ the \emph{$q$-th  nilpotent quotient} $N_qG(\UU)$ of $G(\UU)$.
The following is proved in \cite[Thm.~2.23]{AMY} for $2$--dimensional cut-diagrams, and is also known for the $1$--dimensional case (see Remark \ref{rem:dim1} below). 
\begin{theo}\label{thm:alaChen} 
Let $\UU$ be a cut-diagram over $M=M_1\cup \cdots \cup M_n$, and let $R_i$ and $\{w_{ij}\}$ be a choice of meridian and system of loop-longitudes for each $i$.
We have the following presentation for $N_qG(\UU)$ for each $q\ge 1$:\\[-0.35cm]
\[ N_qG(\UU) = 
\Big\langle R_1,\ldots,R_n\ \ \Bigg|
  \begin{array}{l}
    F_q\, ; \, \big[R_i,\eta_q(w_{ij})\big]\textrm{ for all $i,j$.}
  \end{array}
\Big\rangle.
\]
\end{theo}
\noindent 
Here $F$ is the free group generated by the meridians $R_i$, and $\eta_q$ is the $q$-th \emph{Chen map}, which we next define. 
These maps are homomorphisms from the free group $\overline{F}$, generated by \emph{all} regions $R_{ij}$ of $\UU$ (including the distinguished regions $R_i$), to the free group $F$; they provide a \lq rewriting algorithm\rq\, into a word in the chosen meridians. 

Defining these Chen maps requires the choice of an auxiliary combinatorial data:
\begin{defi}\label{def:ontheroadagain}
A \emph{road network} $\alpha$  for $\UU$ is the choice, on each component $M_i$ of
$M$, of a collection of generic oriented paths $\alpha_{ij}$, called \emph{roads}, running from the basepoint $p_i$ to a point in the interior of each region $R_{ij}$ of $M_i$. 
\end{defi}
\begin{rem}\label{rem:ontheroadagain}
In the case of a $1$--dimensional cut-diagram, there is a canonical road network, as follows.
For each component $M_i$ of $M$, each arc $\alpha_{ij}$ runs from the basepoint $p_i$ to the region $R_{ij}$ following the orientation. Note in particular that when doing so, the road to some region $R$ contains all roads associated with the regions previously met when traveling from the basepoint to $R$. 
The $2$--dimensional case is very different from this point of view, and there is, in general, no canonical choice of road network. 
\end{rem}

\begin{defi}\label{def:alaChen} 
Let $\alpha$ be a road network for $\UU$. 
The \emph{Chen maps} are a family of homomorphisms  
 $\eta_q^\alpha:\bar{F}\rightarrow F$  defined inductively by setting, for every $i,j$ and $q\ge1$:
 \begin{gather*}
   \eta_1^\alpha(R_{ij}):=R_i,\\
\eta_{q+1}^\alpha(R_i):=R_i\ \textrm{ and } \
\eta_{q+1}^\alpha(R_{ij}):=\eta_q^\alpha(v_{ij})^{-1} R_i \eta_q^\alpha(v_{ij}), 
\end{gather*}
\noindent where $v_{ij}=\widetilde{w}_{\alpha_{ij}}\in \overline{F}$ is the word representing the road $\alpha_{ij}$ of $\alpha$ running to region $R_{ij}$.
\end{defi}

\begin{rem}\label{rem:oulade}
 In the statement of Theorem \ref{thm:alaChen} we omitted the road network in the notation for the Chen map, because the result holds for \emph{any} choice of road network.
\end{rem}

\begin{rem}\label{rem:dim1}
The proof of Theorem \ref{thm:alaChen} is given in the $2$--dimensional case in \cite{AMY}, but the arguments also apply to $1$--dimensional cut-diagrams, where the proof is in fact much simpler, see Remark \ref{rem:ontheroadagain}.
This type of result was first given for links (i.e. topological cut-diagrams on circles) by Milnor \cite{Milnor2}, building on the work of Chen \cite{Chen} where these maps $\eta_q$ first implicitly appear. The string link case (topological cut-diagrams on intervals) is due to Habegger and Lin \cite{HL}: note that in this case no commutation relations are involved in the presentation, since there are no loop-longitude. The case of welded links and string links 
(general cut-diagrams on circles and intervals) was first obtained in \cite{Chrisman} and \cite{ABMW}, respectively. 
\end{rem}

A number of properties of these Chen maps are given in \cite[Section~6]{AMY}. 
We only recall here two statements that will be used in this paper. 
The first one is a homotopy invariance result, and the second one provides an explicit way to relate two presentations of $N_q G(\UU)$ by using Theorem \ref{thm:alaChen} with two different road networks. 
\begin{lem}\label{eq:alaChen}\cite[Prop.~6.8]{AMY}\,
Let $\alpha$ be  any choice of road network for $\UU$. 
If $\gamma$ and $\gamma'$ are two homotopic paths on $M$ relative to boundary, then $\eta_q^\alpha(w_\gamma) \equiv \eta_q^{\alpha}(w_{\gamma'})$ mod $F_q\cdot V_{(q)}$, 
where $V_{(q)}$ is the normal closure of all relations $[R_i, \eta_q^\alpha(w_{ij})]$ in $F$.
\end{lem}
\begin{lem}\label{lem:rewrite}\cite[Lem.~6.10~\&~Cor.~6.11]{AMY}\,
Let $\alpha$ and $\alpha'$ be two choices of road networks for $\UU$ and let $\{R_i\}$ and $\{R'_i\}$, respectively, be the corresponding sets of meridians.\footnote{For each component, the road network is based at some region, which is the corresponding meridian.}
Consider, for each $i$, the road $a_i$ in $\alpha'$ running to $R_i$, and let $\nu_i=\widetilde{w}_{a_i}\in \overline{F}$ be the associated word.  
For any $w\in \overline{F}$, and any $q\ge 1$, the word $\eta^{\alpha'}_q(w)$ in $\{R'_i\}$ is obtained from the word $\eta^{\alpha}_q(w)$ in $\{R_i\}$ by replacing each $R_i$ by $\eta^{\alpha'}_q(\nu_i)^{-1}R'_i\eta^{\alpha'}_q(\nu_i)$.
\end{lem}

\subsection{Nilpotent peripheral system}

We now review the notion of (nilpotent) peripheral system for cut-diagrams. 
Since this will only be used here in the $1$--dimensional case, we restrict ourselves to this simpler case. So let $\CC$ be a cut-diagram over a compact oriented $1$-manifold $X=X_1\cup \cdots \cup  X_n$. 

As above, pick for each $i$ a basepoint $p_i$ in the interior of some region of $X_i$ (recall  from Remark \ref{rem:long1} that this choice is canonical for each interval component of $X$); this specifies a choice $\{R_i\}$ of meridians for $\UU$.
For each $i$ such that $X_i\cong S^1$, consider the associated loop-longitude $\{\lambda_i\}$ for $X_i$ as in Remark \ref{rem:long1}. 
For each $i$ such that $X_i\cong  [0,1]$, we likewise consider the word $\lambda_i:= w_{\gamma^\partial_{i}}$ in $G(\UU)$ associated with the arc $\gamma^\partial_{i}$ running along $X_{i}$.

The following is a reformulation of \cite[Def.~2.15]{AMY} in the $1$--dimensional case. 
\begin{defi}\label{def:pereiferic}
A \emph{peripheral system} for $\UU$ is the data $\big(G(\UU);\{R_i,\lambda_{i}\}\big)$, 
associated with a choice $\{R_i\}$ of meridians. \\
Two peripheral systems 
$\big(G;\{R_i,\lambda_{i}\}\big)$ and $\big(G';\{R'_i,\lambda'_{i}\}\big)$ are \emph{equivalent} if there is a group isomorphism $\rho:G\rightarrow G'$ and elements $\{g_i\}$ of $G'$ such that, for all $i,j$: 
\[
\rho\left(R_i\right)=g_i^{-1}  R'_i g_i\quad;\quad \rho\left(\lambda_{i}\right) =\left\{
 \begin{array}{cl}
  g_i^{-1} \lambda'_{i} g_i & \textrm{if $X_i\cong  S^1$}\\
  \lambda'_{i} & \textrm{if $X_i\cong  [0,1]$}\end{array}.\right. \]
\end{defi}
Note that two equivalent cut-diagrams have equivalent peripheral systems. 

For each $q\ge 1$, we can likewise define the \emph{$q$-th nilpotent peripheral system} of $\UU$ as the data, associated to a peripheral system $\big(G(\UU);\{R_i,\lambda_{i}\}\big)$, given by the nilpotent quotient $N_qG(\UU)$ of $G(\UU)$, together with the images of all $R_i$ and $\lambda_{i}$ in $N_qG(\UU)$. Equivalent nilpotent peripheral systems are defined just as in Definition \ref{def:pereiferic}. 

As a continuation to Proposition \ref{prop:ortion}, we note the following easy, yet important fact. 
\begin{prop}\label{prop:retnet}
For a tangle $T$, and for any $q\ge 1$, the (topologically defined) $q$-th nilpotent peripheral system is equivalent to the (diagrammatically defined) $q$-th nilpotent peripheral system of any  associated cut-diagram $\Xi_1(T)$. 
\end{prop}

\section{Cut-concordance and main invariance result} \label{sec:Concordance}

We can now introduce the main new notion of this paper, cut-concordance, which is an equivalence relation for $1$--dimensional cut-diagrams generalizing the classical topological relation of concordance.  

\subsection{Cut-concordance}

Let $X$ be a compact oriented $1$-manifold.
Roughly speaking, two ($1$--dimen\-sio\-nal) cut-diagrams $\UU_0$ and $\UU_1$ over $X$ are cut-concordant if there is a $2$--dimensional cut-diagram over $X\times[0,1]$ which intersects $\Sigma\times\{0\}$ and $\Sigma\times\{1\}$ as  $\UU_0$ and $\UU_1$, respectively. 
\begin{defi}\label{def:kut}
  Two cut-diagrams $\UU_0$ and $\UU_1$ over $X$ are
  \emph{cut-concordant} if there is a cut-diagram $\UU$ over $X\times[0,1]$, and   homeomorphisms 
 $f_\e:  \left(X,\UU_\e\right)\rightarrow \left(X\times \{\e\},\UU\cap\big(X\times\{\e\}\big)\right)$ ($\e\in\{0,1\}$) satisfying the two following compatibility rules:
  \begin{description}
   \item[Orientation] a cut-point $p$ of $\UU_0$ (resp. $\UU_1$) is mapped to an inward-oriented cut-arc endpoint if and only if $p$ has positive sign (resp. negative sign);
    \item [Labeling] the following diagram commutes for $\e\in\{0,1\}$:
\[
\xymatrix@!0 @R=1.5cm @C=5cm{
\big\{\textrm{cut-points of $\UU_\e$}\big\} \ar[r]^(.52){\textrm{labeling}} \ar[d]^(.45){\psi_\e} & \big\{\textrm{regions of
  $\UU_\e$}\big\}\ar[d]^(.45){\phi_\e}\\
\big\{\textrm{cut-arcs of $\UU$}\big\} \ar[r]^(.52){\textrm{labeling}} & \big\{\textrm{regions of $\UU$}\big\}\\
}.
\]
Here $\psi_\e$ maps a cut-point $p$ of $\UU_\e$ to the unique cut-arc of $\UU$ containing $f_\e(p)$; likewise, 
$\phi_\e$ maps a region $r$ of $\UU_\e$ to the unique region of $\UU$ that contains $f_\e(r)$. 
\end{description}
Moreover, if $\p X\neq\emptyset$, we also require that
$\UU\cap\big(\p X\times[0,1]\big)\cong f_\e\big(\UU_0\cap\p X\big)\times[0,1]$.   \\
We say that $\UU$ is a \emph{cut-concordance} between $\UU_0$ and $\UU_1$.
\end{defi}
\begin{convention}
In figures, we shall still represent cut-arcs of a cut-concordance $\UU$ in red, and we will depict in blue the $1$--dimensional cut-diagrams $\UU_\e$ on $X\times \{\e\}$, by identifying $\UU_\e$ and $f_\e(\CC_\e)$. 
\end{convention}

\begin{lem}\label{lem:mings}
Two equivalent cut-diagrams are cut-concordant.  
\end{lem}
\begin{proof}
Since two equivalent cut-diagrams are related by a sequence of the three topological moves R1, R2 and R3 of Figure \ref{fig:IsotopyDim1}, which are local, it suffices to give a local description of the cut-concordance associated with each of these moves. 
Figure \ref{fig:svmoves} represents (portions of) cut-concor\-dan\-ces between  cut-diagrams that differ, from left to right, by a topological move R1, R2 and R3 move. 
\begin{figure}[h]
  \includegraphics[scale=0.67]{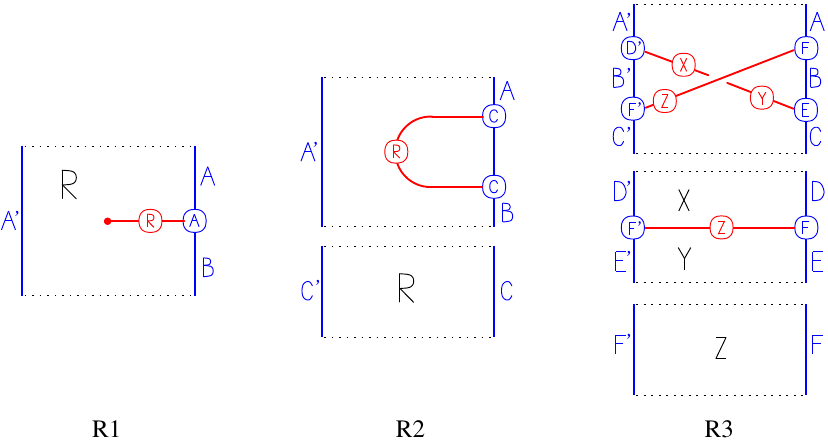}
  \caption{Cut-concordances for moves R1, R2 and R3.\\{\footnotesize Any coherent choice of orientations and signs is allowed}}
  \label{fig:svmoves}
\end{figure}
\end{proof}

Cut-concordance defines a natural equivalence relation on $1$--dimensional cut-diagrams, thus providing a coherent theory for studying these objects.  
Moreover, this notion generalizes the topological notion of concordance in the following sense.
\begin{prop}\label{prop:ConcordanceSameSame}
Two topological cut-diagrams associated to two concordant tangles are cut-concordant.
\end{prop}
\begin{proof} 
Let  $\UU_0$ and $\UU_1$ be two $1$--dimensional topological cut-diagrams representing two tangles $T_0$ and $T_1$ in $B^3$. Let $W\subset B^3\times [0,1]$ be a concordance between $T_0$ and $T_1$. 
Projecting onto $B^2\times [0,1]$ generically produces a surface diagram $D_W$ for $W$  such that $D_W\cap (B^2\times \{\varepsilon\})$ is a diagram for $T_\varepsilon$ ($\varepsilon=0,1$).
Following the procedure associating a topological cut-diagram to a surface diagram (Section \ref{sec:kut2}), we obtain a cut-concordance between two $1$--dimensional cut-diagrams that are equivalent to $\UU_0$ and $\UU_1$, respectively. The conclusion then follows by Lemma \ref{lem:mings}.  
\end{proof}
In particular, two topological cut-diagrams associated to ambient isotopic tangles, relative to boundary, are cut-concordant.

However, in strong contrast with the usual  notion of concordance, it turns out that two $1$-component cut-diagrams are always cut-concordant. Indeed, we have the following.
\begin{lem}\label{lem:cutslice}
Let $X$ be either a circle or an interval. Any cut-diagram over $X$ is cut-concordant to the empty one. 
\end{lem}
\begin{proof}
Let $\UU_0$ be a cut-diagram over $X$.  
Then $\UU:=\UU_0\times [0,1/2]\subset \left(X\times [0,1]\right)$ simply defines a cut-concordance between $\UU_0$ and the empty cut-diagram, as follows. 
Each arc in $\UU$ is oriented according to the Orientation rule of Definition \ref{def:kut}, and is labeled by $R$, the unique region of $\left(X\times [0,1]\right)\setminus \UU$. This trivially fulfills the Labeling rule of Definition \ref{def:kut}, since $\phi_0$ maps all regions of $\UU_0$ to $R$.
\begin{figure} 
 \includegraphics[scale=0.6]{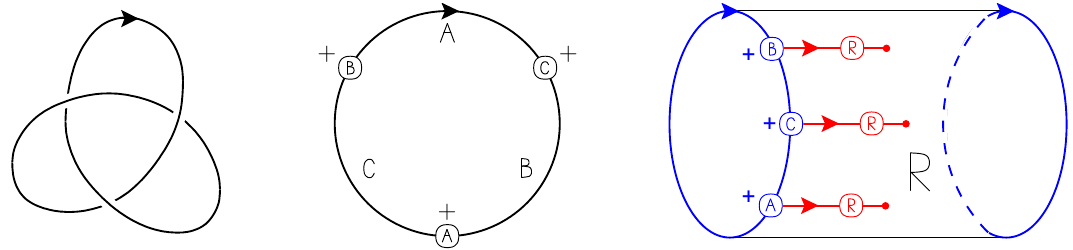}
  \caption{A trefoil diagram (left), the associated cut-diagram (center) and a cut-concordance between this cut-diagram and the empty one (right).}
\label{fig:tefolle}
\end{figure}
\end{proof}
An example is given in Figure \ref{fig:tefolle}, showing that the trefoil is cut-concordant to the unknot.

\subsection{Cut-concordance invariance result}\label{sec:ompletementfoireux}

It is well-known that equivalent classes of nilpotent peripheral systems and Milnor invariants
of link in $S^3$ are concordance invariants \cite{casson}.
This remains true for cut-diagrams. Indeed, we have the following.
\begin{theo}
\label{th:NilpotentConcordance}
  Equivalence classes of nilpotent peripheral systems for $1$--dimensional cut-diagrams are invariant under cut-concordance.
\end{theo}

Before proceeding with the proof of this result in Section \ref{sec:onsamusebiennon}, 
let us note the following consequence. 
Recall that Milnor invariants are  (classical) link invariants originally defined in \cite{Milnor,Milnor2}, which are in some sense  \lq higher order linking numbers\rq \, extracted from the nilpotent peripheral systems of links. These invariants were adapted to string links by Habegger and Lin, under the form of certain automorphisms of the nilpotent quotients of the free group \cite{HL}. 
Both constructions were later extended to welded links and string links  \cite{ABMW,Chrisman,MWY}, that is to $1$--dimensional cut-diagrams over a union of circles or intervals. As their classical counterparts, these \emph{welded Milnor invariants} only depend on the nilpotent peripheral systems.  
\begin{cor}\label{cor:MilnorConcordance}
  Welded Milnor invariants are invariant under cut-concordance.
\end{cor}
\begin{rem}\label{rem:ilnor}
If $T_1$ and $T_2$ are two tangles such that $\Xi_1(T_1)$ and $\Xi_1(T_2)$ are cut-concordant, then $T_1$ and $T_2$ have same Milnor invariants by Proposition \ref{prop:retnet}. In this sense, we can say that classical Milnor invariants are invariant under cut-concordance.
\end{rem}

\subsection{Proof of Theorem \ref{th:NilpotentConcordance}}\label{sec:onsamusebiennon}

Suppose that $\UU$ is a cut-concordance between two cut-diagrams $\UU_0$ and $\UU_1$ over a compact oriented $1$-manifold $X=X_1\cup \cdots \cup X_n$. 
We will freely use the notation from Definition \ref{def:kut}. 

\subsubsection{Nilpotent quotients}\label{sec:ool}
Let us first show that, for any given $q\ge 1$, the nilpotent quotients $N_qG(\UU_0)$ and $N_qG(\UU_1)$ are both isomorphic to $N_qG(\UU)$. 

Our strategy is as follows. 
First, we modify $\UU_0$, $\UU_1$ and $\UU$ by topological moves, such that for each component of $X$, there is a one-to-one correspondence between the regions of $\UU_\varepsilon$ and the regions of $\UU$ intersecting $\CC_\varepsilon$ ($\varepsilon=0,1$). 
Next, we pick \lq compatible\rq\, road networks for $\UU_\varepsilon$ and $\UU$, and representatives for the loop-longitudes, such that the Chen maps give the exact same presentations for $N_qG(\UU_\varepsilon)$ and $N_qG(\UU)$ by Theorem \ref{thm:alaChen}. 
\medskip

We first focus on regions of $\CC_0$ which intersects $\partial X$.  
More precisely, for each $i$ such that $X_i\cong [0,1]$, consider the region $R$ of $\UU$ containing $\{0\}\times [0,1]$. The region $R$ intersects $\UU_0$ and $\UU_1$ at regions  $A$ and $B$ respectively, see the left-hand side of Figure \ref{fig:trick}.  
Apply the ($1$--dimensional) topological move R1 of Figure \ref{fig:IsotopyDim1} at both regions $A$ and $B$, and modify $\UU$ accordingly by adding $R$-labeled cut-arcs, as shown in the center of Figure \ref{fig:trick}. We still denote by $\CC_0$ and $\CC_1$ the resulting $1$--dimensional cut-diagrams, which are indeed equivalent to the previous ones. The resulting $2$--dimensional cut-diagram clearly still defines a cut-concordance, with group isomorphic to $G(\CC)$; we shall hence still denote it by $\UU$. 
Finally, apply move T5 to $\CC$ as represented on the right-hand side of Figure \ref{fig:trick}. The resulting cut-concordance between $\CC_0$ and $\CC_1$ (still denoted $\UU$) now has a new region $R'$  containing $\{0\}\times [0,1]$, 
such that $R'\cap \UU_\e$ is a single region of $\UU_\e$  ($\e\in\{0,1\}$). 
By Proposition \ref{prop:iniscent}, it is enough to show the result for these new $\CC_0$, $\CC_1$ and $\CC$. 
\begin{figure} 
  \[
    \begin{array}{ccccc}
      \dessin{1.6cm}{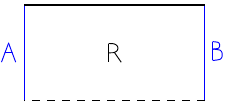} & \xrightarrow{\text{\large $\textrm{ }$}} & 
            \dessin{1.6cm}{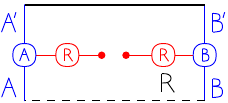} & \xrightarrow{\text{\large T5}}
           &  \dessin{1.6cm}{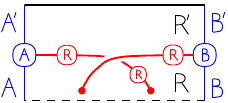}
    \end{array}
    \]  \caption{Modifying $\CC$ near the boundary component $\{0\}\times [0,1]$.\\{\footnotesize 
The thick black line is the boundary $\{0\}\times [0,1]$; cut-point signs and cut-arc orientations can be chosen in any coherent way 
}}
\label{fig:trick}
\end{figure}
\medskip

We next focus on regions of $\UU_0$ that are incident to two cut-points, and begin with setting some notation. 
Denote by $\mathcal{R}_0$ and $\mathcal{R}_\C$ the set of regions of $\UU_0$ and $\UU$, respectively. 
Let $\mathcal{R}^0_\UU\subset \mathcal{R}_\UU$ be the set of regions of $\UU$ intersecting 
$\phi_0(\mathcal{R}_0)$. We shall say that regions in $\mathcal{R}^0_\UU$ are \emph{adjacent to $\UU_0$}. 
Following our strategy, we further modify the cut-concordance $\UU$ into an equivalent one, which we shall still denote by $\UU$, such that the map $\phi_0: \mathcal{R}_0\rightarrow \mathcal{R}_{\UU}$ provides a bijection between $\mathcal{R}_0$ and $\mathcal{R}^0_{\UU}$.  
\\
Consider a region $A$ of $\CC_0$ which is incident to two cut-points, and denote by $R$ the adjacent region of $\CC$. Note that $A$ is any region of $\CC_0$, except those regions containing a point in $\partial X_i\times \{0\}$.  
Apply the following \emph{duplication process}, which is illustrated in Figure \ref{fig:theonlyidea}: 
\begin{description}
\item[Step 1] Apply move T1 in $R$ near $A$, introducing a small $R$-labeled cut-arc with two internal vertices; 
\item[Step 2] For each $R$-labeled cut-arc of $\UU$ that is incident to an $A$-labeled cut-point of $\CC_0$, apply move T3 to insert a small kink as shown in the middle of the figure; 
\item[Step 3] Apply two T5 moves as illustrated: this operation splits $R$ into two regions, denoted by $R$ and $R'$, with $R'$ now intersecting $A$ instead of $R$. Following Section \ref{sec:kut2.2}, we can relabel by $R'$ all $R$-labeled cut-arcs that are incident to an $A$-labeled cut-point of $\CC_0$ so that the labeling rules of Definition \ref{def:cut} are fulfilled.
\end{description}
\begin{figure} 
  \[
    \begin{array}{ccccc}
      \dessin{5.5cm}{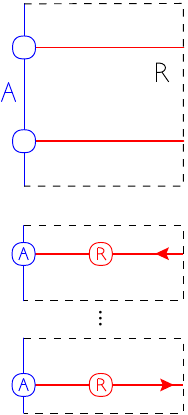} & \xrightarrow{\text{\large Step 1 \& 2}} & 
            \dessin{5.5cm}{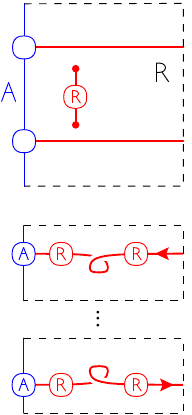} & \xrightarrow{\text{\large Step 3}}
           &  \dessin{5.5cm}{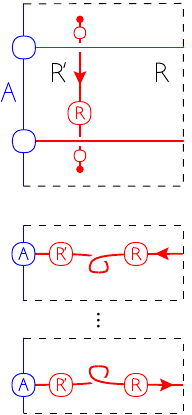}
    \end{array}
    \]  \caption{Duplicating a region of $\UU$ adjacent to $\UU_0$.\\{\footnotesize 
    The figures represents portions of a neighborhood of $\CC_0$ where the local moves are performed, and only relevant labels and orientations are indicated 
}}
\label{fig:theonlyidea}
\end{figure}
Using this duplication process repeatedly at all regions of $\UU_0$ that are incident to two cut-points, we obtain a cut-concordance between $\UU_0$ and $\UU_1$, still denoted by $\UU$, such that $\phi_0: \mathcal{R}_0\rightarrow \mathcal{R}^0_{\UU}$ is bijective 
(that is, each region of $\UU_0$ has a unique adjacent region in $\UU$), and all $\phi_0(A)$-labeled cut-arcs of $\UU$ adjacent to $\UU_0$ intersect $\UU_0$ at $A$-labeled cut-points.
Note that by Proposition \ref{prop:iniscent}, the nilpotent quotient $N_qG(\UU)$ is unchanged by this duplication process. 

Denote by $\overline{F}_0$ and $\overline {F}$ the free groups generated by all regions of $\UU_0$ and $\UU$, respectively. 
For each $i$, 
fix a basepoint $p_i$ in the interior of some region $r_i^0$ of $\CC_0\cap X_i$ (recall from Remark \ref{rem:long1} that this choice is canonical if $X_i\cong [0,1]$).  
Since $p_i$ also lies in a region $R_i^0=\phi_0(r^0_i)$ on the $i$-th component of $\UU$, we can identify these regions $r_i^0=R_i^0$, and denote by $F$ the free group generated by these $n$ elements. 

We next fix road networks for $\UU_0$ and $\UU$.
As noted in Remark \ref{rem:ontheroadagain}, there is a canonical road network $\alpha(\UU_0)$ for $\UU_0$.  
We then pick a road network $\alpha_0(\UU)$ for $\UU$ which is \lq compatible\rq\, with $\alpha(\UU_0)$: each road $\alpha_{ij}$ running to a region $R_{ij}\in \mathcal{R}^0_\CC$ 
is a small push-in of the corresponding road $\alpha^0_{ij}$ in $\alpha(\UU_0)$ running to $\phi_0^{-1}(R_{ij})$. 
Therefore the word $\widetilde{w}_{\alpha_{ij}}\in \overline{F}$ associated with $\alpha_{ij}$ is exactly obtained from the word in $\overline{F_0}$ associated with $\alpha^0_{ij}$ by replacing each letter by its image by $\phi_0$.
Consequently we have :
\begin{claim}\label{claim:entine}
 For all $q\ge 1$, and any $r\in \overline{F}_0$, we have $\eta^{\alpha(\UU_0)}_{q}(r) = \eta^{\alpha_0(\UU)}_{q}(\phi_0(r))\in F$. 
\end{claim}
Now, each loop-longitude of $\UU_0$ is represented by a (based) copy of the corresponding $S^1$-component of $X$. This yields, after a small push-in, a (based) loop $\lambda^0_{i}$ in $X_i\times [0,1]$ representing a loop-longitude of $\CC$. Claim \ref{claim:entine} implies directly that these systems of loop-longitudes for $\UU_0$ and $\UU$, have the same image in $F$ by the Chen maps.
The statement of Theorem \ref{thm:alaChen} then provides the exact same presentation for $N_qG(C_0)$ and for $N_qG(C)$ for all $q$.
We hence have an isomorphism 
 \[ \rho_0: N_qG(\UU_0)\stackrel{\simeq}{\longrightarrow} N_qG(\UU)\quad ; \quad r_i^0\mapsto R_i^0. \]

Now, applying the same duplication process at all regions of $\CC_1$ that are adjacent to two cut-points, we likewise obtain an isomorphism 
$\rho_1: N_qG(\UU_1)\stackrel{\simeq}{\longrightarrow} N_qG(\UU)$. 
We can summarize the situation so far as follows: 
\begin{prop}\label{rem:Stallings}
Let $\CC$ be a cut-concordance between two $1$--dimensional cut-diagrams $\UU_0$ and $\UU_1$. 
Then for all $q\ge 1$ we have isomorphisms 
 \[ N_qG(\UU_0)\stackrel{\simeq}{\longrightarrow} N_qG(\UU)\stackrel{\simeq}{\longleftarrow} N_qG(\UU_1). \]
\end{prop}
As outlined in the introduction, this provides a combinatorial proof of Stallings' Theorem \cite[Thm.~5.2]{Stallings} in the case of a topological concordance between two tangles, and a generalization to cut-concordances.
 
 \subsubsection{Peripheral systems}\label{sec:on}

It now remains to prove that this isomorphism 
\[ \rho:=\rho_1^{-1}\circ \rho_0:  N_qG(\UU_0)\stackrel{\simeq}{\longrightarrow} N_qG(\UU_1) \]
provides an equivalence of nilpotent peripheral systems, in the sense of Definition \ref{def:pereiferic}. 

By Section \ref{sec:ool}, $N_qG(\UU)$ is given two presentations, which coincide with a presentation for $N_qG(\UU_0)$ and $N_qG(\UU_1)$, respectively. 
In particular, for $\e\in\{0,1\}$, the chosen meridians $r^\e_i$ of $\UU_\e$ coincide with  regions $R^\e_i$ of $\CC$ that are adjacent to $\UU_\e$. 
Moreover, each of these two presentations is associated with a road network $\alpha_\e(\UU)$ for  the cut-concordance $\UU$ ($\e=0,1$).
Consider for each $i$ the road $a_i$ in $\alpha_1(\UU)$ running from $R^1_i$ to $R^0_i$ in $\CC$. By Lemma \ref{lem:rewrite}, we have 
\[ \rho(R^0_i) = g_i^{-1} R^1_i g_i\textrm{, \,  where $g_i:=\eta^{\alpha_1(\UU)}_q(\widetilde{w}_{a_i})$.}  \]

We next observe that these conjugating elements $g_i$ also relate the systems of loop-longitu\-des of $\UU_0$ and $\UU_1$. 
As noted above, each loop-longitude of $\UU_\e$ is represented by a (based) copy of the corresponding $S^1$-component of $X$ and gives, after a small isotopy, a loop $\lambda^\e_{i}$ in $X_i\times [0,1]$ based at region $R^\e_i$. Hence $\lambda^0_{i}$ is homotopic, relative to basepoint, to $a_i^{-1}\lambda^1_{i}a_i$: the result then follows from the homotopy invariance property of the Chen map $\eta^{\alpha_1(\UU)}_q$ of Lemma \ref{eq:alaChen}. 

Finally, for each $i$ such that $X_i\cong  [0,1]$, the corresponding peripheral element $(\lambda_i)_\e$ of $\UU_\e$ is represented by an arc $\iota^\e_{i}$ which is a copy of $X_i$ ($\e=0,1$). 
Since by assumption we have that $\UU\cap\big(\p X\times[0,1]\big)\cong \big(\UU_0\cap\p X\big)\times[0,1]$, we immediately have that $\iota^0_{i}$ and $\iota^1_{i}$ are homotopic relative to boundary. Lemma \ref{eq:alaChen} then implies that peripheral  elements $(\lambda_i)_0$ and $(\lambda_i)_1$ are in bijection through the isomorphism $\rho$. 

This completes the proof that $\UU_0$ and $\UU_1$ have equivalent nilpotent peripheral systems. 

\section{Welded concordance, self-virtualization, and cut-diagrams}\label{sec:arrementnimp}

In this final section, we investigate the relationship between cut-concordance and other equivalence relations that appear in diagrammatic knot theory. 

\subsection{Welded concordance vs. cut-concordance}\label{sec:4.1}

The notion of welded concordance was introduced in \cite{BC,Gaudreau}, as a diagrammatic  extension of the classical notion of concordance to welded tangles.
Two welded tangles $L$ and $L'$ are \emph{welded concordant} 
 if these virtual diagrams are related by a sequence of welded Reidemeister moves, and the birth/death and saddle moves of Figure \ref{fig:concmoves} such that, for each $i$, the number of birth/death moves used to deform the $i$-th component of $L$
into the $i$-th component of $L'$ is equal to the number of saddle moves.
\begin{figure}[!h]
 \includegraphics[scale=0.9]{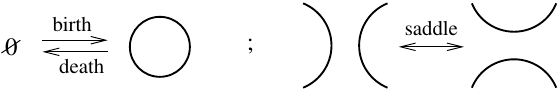}
 \caption{The birth/death and saddle moves}\label{fig:concmoves}
\end{figure}

Gaudreau proved that any welded (long) knot is welded concordant to the trivial one \cite[Thm.~5.1]{Gaudreau}. This is reminiscent of Lemma \ref{lem:cutslice} on cut-concordance, which is indeed a generalization of Gaudreau's result by the following. 
\begin{lem}\label{prop:wctocut}
 Two $1$--dimen\-sio\-nal cut-diagrams associated to welded concordant welded tangles are cut-concordant.
 \end{lem}
 \begin{proof}
 Consider two welded concordant diagrams: there is a sequence of welded moves, birth/death and saddle moves deforming one into the other. 
Regarding each diagram in this sequence as a decorated\footnote{Here by decorated we mean that each transverse double point is decorated either as classical or virtual crossing. } proper immersion of some $1$-manifold $X$ in the $2$-disk $D^2$, we can take the \lq trace\rq\, of this sequence. By the combinatorial assumption on the number of birth/death and saddle moves, this trace yields an immersed surface $S:X\times [0,1]\hookrightarrow D^2\times [0,1]$,  with lines of double points decorated either as usual lines of double points or as a \lq virtual\rq\, line with no over/under information. 
This naturally produces a $2$--dimensional cut-diagram $\UU_S$ on $X\times [0,1]$, by considering the preimage of $S$ as in Section \ref{sec:kut2.2} and ignoring virtual lines of double points. It is easily checked that $\UU_S$ satisfies the compatibility rules of Definition \ref{def:kut}.
\end{proof}

\subsection{Self-virtualization and cut-diagrams}\label{sec:4.2}

Link-homotopy is an equivalence relation on links, and more generally on tangles,  introduced by Milnor \cite{Milnor2}. It is generated by ambient isotopy and \emph{self-crossing changes}, which are local moves exchanging the relative position of two strands of a same component, see Figure \ref{fig:SCSV}. 
\begin{figure}[h]
    \[
  \begin{array}{ccc}
     \dessin{1.5cm}{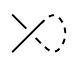}\ \longleftrightarrow
      \dessin{1.5cm}{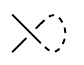}
      & &
     \dessin{1.5cm}{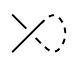}\ \stackrel{\textrm SV}{\longleftrightarrow}
      \dessin{1.5cm}{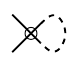} 
      \end{array} \]
  \caption{A self-crossing change (left) and a self-virtualization (right)}
  \label{fig:SCSV}
\end{figure}

When considering the wider context of welded tangles, it turns out that the relevant analogue of link-homotopy is the \emph{sv-equivalence}. Two welded tangles are sv-equivalent  if these virtual diagrams are related by welded Reidemeister moves and the \emph{self-vitualization (SV) move}, which is the local move that virtualizes a classical crossing between two strands of a same component, or vice-versa, see Figure \ref{fig:SCSV}.   Note that a self-crossing change is realized by two self-virtualization moves. 

\subsubsection{Reduced cut-concordance}\label{sec:redcut}\label{sec:4.2.1}

Considering the effect of a self-virtualization move at the cut-diagram level, leads to the following generalization of cut-concordance (see Lemma \ref{lem:svtocut} below). 
\begin{defi}\label{def:svcut}
 Let $\UU_0$ and $\UU_1$ be two cut-diagrams over a compact oriented $1$-manifold $X$. 
 Then $\UU_0$ and $\UU_1$ are \emph{reduced cut-concordant} if there is a cut-diagram $\CC$ over $X\times [0,1]$, satisfying the two compatibility rules of Definition \ref{def:kut}, but where 
 the second labeling rule of Definition \ref{def:cut} is relaxed as follows:
 \begin{itemize}
  \item[2')] a cut-arc containing a univalent vertex in a region of the $i$-th component of $\CC$, is labeled by a region of this $i$-th component. 
  In figures, such internal vertices are represented by a $\oslash$, see Figure \ref{fig:whitehead}  for an example.
  \end{itemize}
 We say that $\CC$ is a \emph{reduced cut-concordance} between $\UU_0$ and $\UU_1$.
\end{defi}
Figure \ref{fig:whitehead} shows that the Whitehead link is reduced cut-concordant to the unlink. 
\begin{figure}[h]
  \includegraphics[scale=0.7]{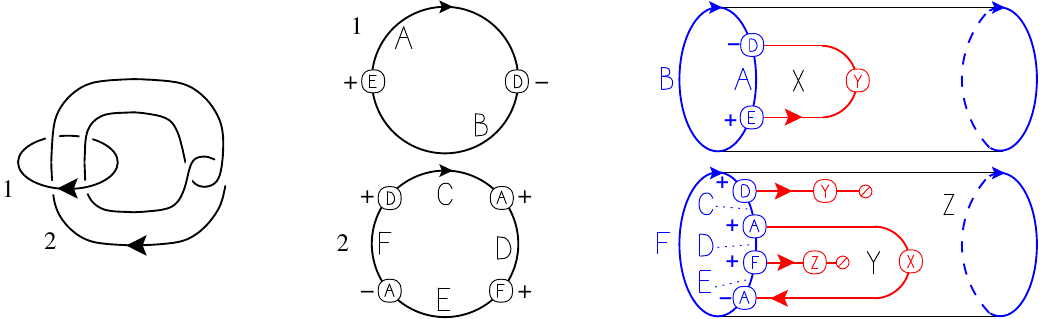}
  \caption{A diagram of the Whitehead link (left), the associated cut-diagram (center) and a reduced cut-concordance between this cut-diagram and the empty one (right)}
  \label{fig:whitehead}
\end{figure}
Notice that the cut-arcs incident to the $\oslash$-vertices are not labeled by the region containing them. As a matter of fact, the Whitehead link is not cut-concordant to the unlink: an obstruction is for example given by Milnor invariant $\mu(1122)$, see Remark \ref{rem:ilnor}. 

 \begin{lem}\label{lem:svtocut}
  Two $1$--dimen\-sio\-nal cut-diagrams associated to sv-equivalent welded tangles are reduced cut-concordant.
 \end{lem}
 \begin{proof}
This can be shown in a similar way as Lemma \ref{prop:wctocut}. 
But since sv-equivalence is generated by local moves only, we can also proceed as in the proof of Lemma \ref{lem:mings}. Indeed, having already given in Figure \ref{fig:svmoves} (portions of) reduced cut-concor\-dan\-ces realizing the topological moves R1, R2 and R3, it only remains to do so for the SV move. This is done in the following figure: 
\[  \includegraphics[scale=0.67]{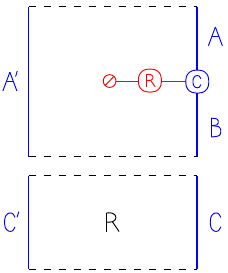} \]
Here, the two represented sheets belong to the same connected component. 
\end{proof}

\subsubsection{Reduced peripheral system}\label{sec:redsys}

The \emph{reduced peripheral system} of cut-diagrams was introduced in \cite{AMY}, building on Milnor's work in the link case \cite{Milnor}. 
We briefly review this notion below, restricting ourselves to the case of a cut-diagram $\UU$ over a $1$--manifold $X$.  

The \emph{reduced group} of a cut-diagram $\UU$ is the quotient $\nR G(\UU)$ of $G(\UU)$ by the normal subgroup generated by all relations $[R,R^g]$, for all meridians $R$ and all $g\in G(\UU)$. 
In the notation of Theorem \ref{thm:alaChen}, we have the following presentation : 
\begin{equation}\label{eq:late}
\nR G(\UU) \cong \left\langle R_1,\ldots,R_n\ \ \left|
  \begin{array}{l}
    \big[R_i,R_i^g\big]\textrm{ for all $i$ and all $g\in F$}\\[.1cm]
    \big[R_i,\eta_q(\lambda_{i})\big]\textrm{ for all $i$ such that $X_i\cong S^1$} 
  \end{array}
\right.\right\rangle.
\end{equation}
\noindent (See \cite[Thm.~4.7]{AMY} for the $2$--dimensional case and \cite[Prop.~18.3.21]{AMY_w} for the $1$--dimensional case.)

\begin{defi}
A \emph{reduced peripheral system} for $\UU$ is the data 
\[ \big( \nR G(\UU),\{R_i,\lambda_{i}\cdot N_i\}\big), \]
associated with a peripheral system 
$\big(G(\UU);\{R_i,\lambda_{i}\}\big)$, 
where for each $i$, $\lambda_i\cdot N_i$ denotes the coset of $\lambda_i$ with
respect to the normal subgroup $N_i$ generated by $R_i$. 
Equivalent reduced peripheral systems are defined as in Definition \ref{def:pereiferic}. 
\end{defi}
As an analogue of Theorem \ref{th:NilpotentConcordance}, we have the following. 
\begin{theo}\label{thm:redisdead}
Equivalence classes of reduced peripheral systems for $1$--dimensional cut-diagrams are invariant under reduced cut-concordance.
\end{theo}
\begin{proof}
The proof follows very closely that of Theorem \ref{th:NilpotentConcordance}. 
More precisely, using (\ref{eq:late}) instead of Theorem \ref{thm:alaChen}, the first half of the proof (Section \ref{sec:ool}) applies essentially verbatim to show that two reduced cut-concordant $1$--dimensional cut-diagrams have isomorphic reduced groups. 
The second half of the proof, showing that this isomorphism induces an equivalence of reduced peripheral systems, follows the same line as Section \ref{sec:on}. The only significant difference here is a \lq reduced\rq\, version of the homotopy invariance Lemma \ref{eq:alaChen}, stated in Lemma \ref{lem:bienlasoupe} below. 
\end{proof}
We make use of the notation of Lemma \ref{eq:alaChen} for this reduced version:  
\begin{lem}\label{lem:bienlasoupe}
Let $\alpha$ be  any choice of road network for 
a ($2$--dimensional) cut-diagram over a surface $\Sigma$. 
If $\gamma$ and $\gamma'$ are two homotopic paths on the $i$-th component $\Sigma_i$ of $\Sigma$, relative to boundary, then 
$\eta_{q}(w_\gamma)\equiv \eta_{q}(w_{\gamma'})\mod F_{q}\!\cdot\!V_{(q)}\!\cdot\!N_i$, 
where $N_i$ denotes the normal subgroup of $F$ generated by $R_i$.
\end{lem}
\begin{proof}
By \cite[Lem.~2.6]{AMY}, the result follows from Lemma \ref{eq:alaChen} and the fact that the desired equality holds when the two paths differ by insertion/deletion of a loop going around a $\oslash$-vertex of $\UU$.  
We thus consider the case where $\gamma$ and $\gamma'$ differ locally as illustrated below, with $\gamma=\gamma_1\cdot\gamma_2$ on the left-hand side of the move. 
\[
\dessin{2cm}{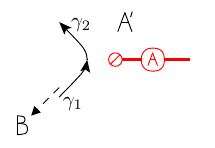} \raisebox{.2cm}{$\longleftrightarrow$} \dessin{2cm}{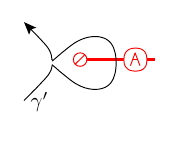}.
\]
Here, $B$ is the region of $\Sigma_i$ where $\gamma$ starts, $A$ and $A'$ being two regions of $\Sigma_i$. 
We then have  
\[w_{\gamma}= B^{s}\widetilde{w}_{\gamma_1}\widetilde{w}_{\gamma_2}
\quad\textrm{and}\quad 
w_{\gamma'}=B^{s-\varepsilon}\widetilde{w}_{\gamma_1}A^{\varepsilon}\widetilde{w}_{\gamma_2},\] 
for some $s\in \Z$ and some $\varepsilon={\pm 1}$. 
Denoting by $\alpha_B$ and $\alpha_A$ the roads of $\alpha$ running to regions $B$ and $A$, respectively, we thus have by definition of the Chen maps: 
\begin{eqnarray*}
\eta^\alpha_{q+1}(w_{\gamma'}) & 
= &
\eta^\alpha_{q}(v_B)^{-1}R_{i}^{s-\varepsilon}\eta^\alpha_{q}(v_B) 
\eta^\alpha_{q+1}(\widetilde{w}_{\gamma_1})
\eta^\alpha_{q}(v_A)^{-1}R_{i}^{\varepsilon}\eta^\alpha_{q}(v_A) 
 \eta^\alpha_{q+1}(\widetilde{w}_{\gamma_2}) \\
  & = & 
  \eta^\alpha_{q}(v_B)^{-1}R_{i}^{s}\left[ R_i^{\varepsilon} , \eta^\alpha_{q}(v_A) \eta^\alpha_{q+1}(\widetilde{w}_{\gamma_1}^{-1}) \eta^\alpha_{q}(v_B^{-1}) \right] \eta_{q}(v_B)\eta^\alpha_{q+1}(\widetilde{w}_{\gamma_1}\widetilde{w}_{\gamma_2}) \\
  & \equiv & \eta^\alpha_{q}(v_B)^{-1}R_{i}^{s}\eta_{q}(v_B)\eta^\alpha_{q+1}(\widetilde{w}_{\gamma_1}\widetilde{w}_{\gamma_2}) = \eta^\alpha_{q+1}(w_{\gamma}) \mod N_i 
\end{eqnarray*}
This concludes the proof.
\end{proof}

A consequence of the invariance Theorem \ref{thm:redisdead} is that the converse of Lemma \ref{lem:svtocut} also holds: 
\begin{cor}\label{lem:cuttosv}
Let $D$ and $D'$ be two welded tangles, and let $\UU_D$ and $\UU_{D'}$ be $1$--dimensional cut-diagrams associated to $D$ and $D'$, respectively.
Then $D$ and $D'$ are sv-equivalent if and only if $\UU_D$ and $\UU_{D'}$ are reduced cut-concordant. 
\end{cor}
\begin{proof}
The \lq only if\rq\, part of the statement corresponds to Lemma \ref{lem:svtocut}. 
For the \lq if\rq\, part,  recall that two reduced cut-concordant $1$--dimensional cut-diagrams have equivalent reduced peripheral systems by Theorem \ref{thm:redisdead}.  
The result follows, since two welded tangles are sv-equivalent if and only if they have equivalent reduced peripheral systems, as proved in \cite[Thm.~18.3.22]{AMY_w}.
\end{proof}

\begin{rem}\label{rem:koolas}
The classification result \cite[Thm.~18.3.22]{AMY_w} was first established in \cite[Thm.~2.1]{AM} in the link case, and in \cite[Thm.~3.11]{ABMW} for string links (see also \cite[\S~9.2]{arrow}).\\
This result also implies that link-homotopy classes of (classical) tangles inject into welded tangles up to sv-equi\-va\-lence (see the bottom-left arrow in the commutative diagram given in the introduction), since the reduced peripheral system is invariant under link-homotopy. 
\end{rem}

\begin{rem}\label{lem:niscatedebernoulli}
It is well-known that concordance implies link-homotopy for classical knotted objects in $3$-space \cite{Giffen,Goldsmith}.
We observe that, likewise, cut-concordance implies reduced cut-concor\-dan\-ce, since a cut-concordance is merely a reduced one with only \lq ordinary\rq\, internal vertices (that is, a reduced cut-concordance with no $\oslash$-vertex). 
Combining this observation with Lemmas \ref{prop:wctocut} and Corollary \ref{lem:cuttosv}, directly implies that welded concordance implies sv-equivalence.
\end{rem}

\bibliographystyle{abbrv}
\bibliography{References}

\end{document}